\documentclass{article}

\usepackage{epsfig}
\usepackage{graphicx}
\usepackage{amsbsy}
\usepackage{amsmath}
\usepackage{amsfonts}
\usepackage{amssymb}
\usepackage{textcomp}
\usepackage{hyperref}
\usepackage{aliascnt}

\newcommand{\mcm}[3]{\newcommand{#1}[#2]{{\ensuremath{#3}}}} 
\newcommand{\sm}{\setminus}
\newcommand{\se}{\subseteq}
\newcommand{\ct}{^\complement}

\mcm{\tuple}{1}{\langle #1 \rangle}
\mcm{\name}{1}{\ulcorner #1 \urcorner}
\mcm{\Nbb}{0}{\mathbb{N}}
\mcm{\Zbb}{0}{\mathbb{Z}}
\mcm{\Qbb}{0}{\mathbb{Q}}
\mcm{\Rbb}{0}{\mathbb{R}}
\mcm{\Cbb}{0}{\mathbb{C}}
\mcm{\Fbb}{0}{\mathbb{F}}
\mcm{\Bcal}{0}{\cal B}
\mcm{\Ccal}{0}{\cal C}
\mcm{\Dcal}{0}{\cal D}
\mcm{\Ecal}{0}{\cal E}
\mcm{\Fcal}{0}{\cal F}
\mcm{\Gcal}{0}{\cal G}
\mcm{\Hcal}{0}{\cal H}
\mcm{\Ical}{0}{\cal I}
\mcm{\Lcal}{0}{\cal L}
\mcm{\Mcal}{0}{\cal M}
\mcm{\Ncal}{0}{\cal N}
\mcm{\Pcal}{0}{{\cal P}}
\mcm{\Scal}{0}{{\cal S}}
\mcm{\Tcal}{0}{{\cal T}}
\mcm{\Ucal}{0}{{\cal U}}
\mcm{\Vcal}{0}{{\cal V}}
\mcm{\Wcal}{0}{{\cal W}}
\mcm{\Xcal}{0}{{\cal X}}
\mcm{\Ycal}{0}{{\cal Y}}
\mcm{\Zcal}{0}{{\cal Z}}
\mcm{\Mfrak}{0}{\mathfrak M}

\mcm{\restric}{0}{\upharpoonright}
\mcm{\upset}{0}{\uparrow}
\mcm{\downset}{0}{\downarrow}
\mcm{\onto}{0}{\twoheadrightarrow}
\mcm{\smallNbb}{0}{{\small \mathbb{N}}}
\DeclareMathOperator{\preop}{op}
\mcm{\op}{0}{^{\preop}}

\newcommand{\itum}{\item[$\bullet$]}

{\begin{array}{c}
\setlength{\unitlength}{1em}}%
{\end{array}}

\usepackage{amsthm}

\newcommand{\theoremize}[2]{\newaliascnt{#1}{thm} \newtheorem{#1}[#1]{#2} \aliascntresetthe{#1}}

\theoremstyle{plain}
\newtheorem{thm}{Theorem}[section]
\theoremize{lem}{Lemma}
\theoremize{sublem}{Sublemma}
\theoremize{claim}{Claim}
\theoremize{rem}{Remark}
\theoremize{prop}{Proposition}
\theoremize{cor}{Corollary}
\theoremize{que}{Question}
\theoremize{oque}{Open Question}
\theoremize{con}{Conjecture}

\theoremstyle{definition}
\theoremize{dfn}{Definition}
\theoremize{eg}{Example}
\theoremize{exercise}{Exercise}
\theoremstyle{plain}

\usepackage{verbatim}

\usepackage{graphicx}
\newcommand{\cev}[1]{\reflectbox{\ensuremath{\vec{\reflectbox{\ensuremath{#1}}}}}}

\DeclareMathOperator{\finn}{fin}
\DeclareMathOperator{\co}{cofin}
\newcommand{\fin}{^{\finn}}
\newcommand{\cofin}{^{\co}}

\title{\scshape The ubiquity of Psi-matroids}
\author{Nathan Bowler \and Johannes Carmesin}

\begin{document}

\maketitle

\begin{abstract}
Solving (for tame matroids) a problem of Aigner-Horev, Diestel and Postle, we prove that 
every tame matroid $M$ can be reconstructed from its canonical tree decomposition into 3-connected pieces, circuits and cocircuits together with  information about which ends of the decomposition tree are used by $M$.

For every locally finite graph $G$, we show that every tame matroid 
whose circuits are topological circles of $G$ and whose cocircuits are bonds of $G$
is determined by the set $\Psi$ of ends it uses, that is, it is a $\Psi$-matroid.

\end{abstract}

\section{Introduction}

There is a canonical matroid associated to any finite graph $G$, whose circuits are the (edge sets of) cycles in $G$ and whose cocircuits are the bonds of $G$ (the minimal nonempty cuts). However, when $G$ is infinite there is no longer a canoncical choice of matroid. Instead, there are at least two very natural choices. The first is the finite cycle matroid $M_{FC}(G)$, defined as above, which is always finitary (all its circuits are finite). The second is the topological cycle matroid $M_{C}(G)$ whose circuits are (edge sets of) topological circles in a suitable compactification $|G|$ of $G$ and whose cocircuits are the finite bonds of $G$~\cite{DiestelBook10}. These topological circles arise naturally in the generalisations of some fundamental theorems from finite to infinite graphs~\cite{RDsBanffSurvey}. The topological cycle matroid is always cofinitary (all its cocircuits are finite).

If $G$ is an infinite planar graph then we cannot hope to find a dual graph $G^*$ such that $M_{FC}(G*) = (M_{FC}(G))^*$ since the first of these is finitary whilst the second need not be. Instead, what happens (under some weak assumptions) is that there is a dual graph $G^*$ such that $M_{C}(G^*) = (M_{FC}(G))^*$ and $M_{FC}(G^*) = (M_{C}(G))^*$~\cite{RD:HB:graphmatroids}.

There is a whole range of matroids sitting in between these two. The new points added in the compactification $|G|$ are called the ends of $G$. For a set $\Psi$ of ends and a countable graph $G$, we very often get a matroid $M_{\Psi}(G)$, called the {\em $\Psi$-matroid} of $G$, whose circuits are the topological circles in $G$ that only use ends from $\Psi$. In particular, we always get a matroid in this way if $\Psi$ is Borel~\cite{BC:psi_matroids}. The extreme cases -- where $\Psi$ is empty or the set of all ends -- give rise to the finite and topological cycle matroids respectively. The $\Psi$-matroids also reflect duality of graphs. If $G^*$ is the dual of $G$ then they have homeomorphic spaces of ends~\cite{BS:dual_ends} and for any Borel set $\Psi$ of ends we have $M_{\Psi\ct}(G^*) = (M_{\Psi}(G))^*$~\cite{DP:dualtrees, BC:psi_matroids} (this specialises to the results above by taking $\Psi$ to be empty or the set of all ends).

From now on we fix some locally finite graph $G$. All $\Psi$-matroids for $G$ have the property that all their circuits are topological circles of $G$ and all their cocircuits are bonds of $G$. We will call matroids with this property $G$-matroids. This paper is concerned with the question:

\begin{que}
 Characterise the $G$-matroids of locally finite graphs $G$.
\end{que}

Not all $G$-matroids are $\Psi$-matroids.

\begin{eg}\label{eg:quircuits}
 Let $Q$ be the graph depicted in \autoref{fig:quircuits}. We say that two topological circuits of $Q$ are {\em equivalent} if their symmetric difference is finite. Let $\Ccal$ be any union of equivalence classes which includes all finite circuits. Then it can be shown that $\Ccal$ is the set of circuits of a $Q$-matroid~\cite{BCP:quircuits}.
\begin{figure}
\begin{center}
 \includegraphics[width=8cm]{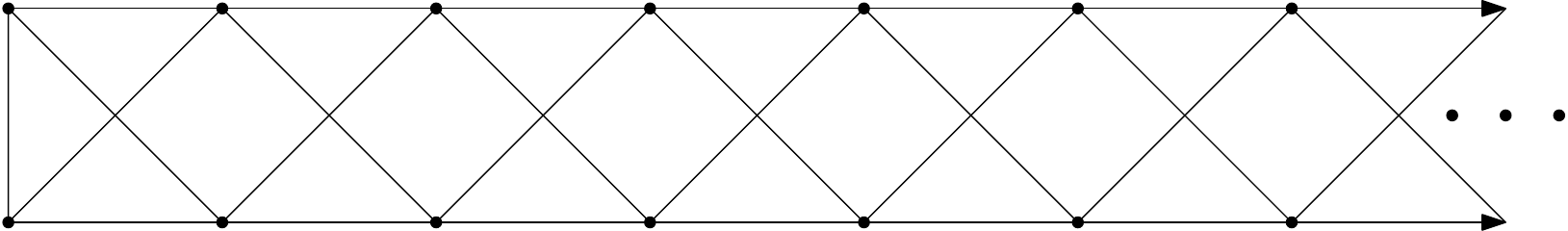}
\end{center} 
\caption{The graph $Q$}	\label{fig:quircuits}
\end{figure}
\end{eg}


In cases like the above example, we can get a large set of badly behaved\footnote{for example, such matroids can be non-binary in the sense that there are 3 circuits whose symmetric difference does not include a circuit} matroids, each specified by a great deal of information `at infinity'. In order to characterise the $G$-matroids of a general graph $G$ we would require a specification of complex information distributed somehow over the ends of $G$, and this is currently intractable. However, there is a simple and natural restriction on the $G$-matroids which makes this problem much more tractable.

We say that a matroid $M$ is {\em tame} if the intersection of any circuit with any cocircuit is finite. Otherwise, we say $M$ is {\em wild}. The concept goes back to Dress~\cite{dress}, though not under this name: he introduced a different notion of infinite matroid to which tameness was fundamental. As recently as 2010 it was suggested that all infinite matroids might necessarily be tame~\cite{matroid_axioms}. Tameness and countability are the only extra restrictions we need to add to Minty's axiomatisation of finite matroids~\cite{Minty} to get an axiomatisation of the class of countable tame matroids: there is no need for an extra axiom requiring the existence of bases~\cite{BC:psi_matroids}. However, there are many wild matroids, as first shown in~\cite{BC:wild_matroids}. In fact, all but 2 of the $Q$-matroids in \autoref{eg:quircuits} are wild~\cite{BCP:quircuits}.

Nevertheless, the class of tame matroids is closed under duality and under taking minors, and includes most natural examples of matroids. For example, all finitary or cofinitary matroids and all $\Psi$-matroids are tame. Tameness gives a natural context for the consideration of representability of infinitary matroids. The notion of thin sums representability was introduced in~\cite{RD:HB:graphmatroids} as a generalisation of representability to infinitary matroids. Although there is a wild matroid which is thin sums representable over $\Qbb$ but whose dual is not thin sums representable over any field~\cite{BC:wild_matroids}, the class of tame thin sums matroids is closed under duality and under taking minors~\cite{THINSUMS}. Moreover, forbidden minor characterisations extend readily from finite to tame thin sums matroids~\cite{BC:rep_matroids}, so that for example a tame matroid is thin sums representable over $\Fbb_2$ if and only if it does not have $U_{2,4}$ as a minor. Finally, tameness plays an 
essential role in the construction of an 
infinitary version of the class of graphic matroids~\cite{BCC:graphic_matroids}, so that the badly behaved $Q$-matroids are not graphic in this sense.

We therefore restrict our attention to the class of tame $G$-matroids, for which we are now able to provide a simple characterisation.

\begin{thm}\label{intpsi}
 The tame $G$-matroids are exactly the $\Psi$-matroids for $G$.
\end{thm}

However the set $\Psi$ need not be Borel. The question of which sets $\Psi$ of ends give rise to matroids is tied to subtle set theoretic questions about determinacy of games \cite{BC:psi_matroids}.

Restricting our attention to tame matroids also allows us to resolve a problem due to Aigner-Horev, Diestel and Postle from~\cite{ADP:decomposition} about the reconstruction of connected matroids from their 3-connected pieces. 
Any finite connected matroid $M$ can be decomposed canonically into a tree of pieces, each of which is 3-connected, a circuit or a cocircuit~\cite{findec1, findec2}. Any two adjacent pieces share only a single edge, and $M$ can be reconstructed from this tree by taking 2-sums along all these edges.

Recall that for any two matroids $M_1$ and $M_2$ that share only one edge, the \emph{2-sum} $M_1 \oplus_2 M_2$ of $M_1$ and $M_2$ is the matroid whose edge set is the symmetric difference of those of $M_1$ and $M_2$ and whose circuits are of the following 3 types: circuits of $M_1$ avoiding $e$, circuits of $M_2$ avoiding $e$, and symmetric differences of an $M_1$-circuit containing $e$ with an $M_2$-circuit containing $e$. This construction is associative, in the sense that if $M_1$ and $M_2$ meet in only one edge and $M_2$ and $M_3$ meet in only one edge, and $M_1$ has no edge in common with $M_3$, then $(M_1 \oplus_2 M_2) \oplus_2 M_3 = M_1 \oplus_2 (M_2 \oplus_2 M_3)$. Because of this associativity, it doesn't matter in what order we take the 2-sums at the edges of the tree: we always get back the original matroid $M$.

Aigner-Horev, Diestel and Postle partially extended this result to infinite matroids: they were able to show that there is such a canonical tree decomposition of any connected matroid~\cite{ADP:decomposition}. It is a little surprising that the structure obtained is a genuine graph-theoretic tree, rather than one of the more order-theoretic or topological notions of infinite tree discussed in \cite{nikiel}.
However, reconstruction of the original matroid from this tree is not so straightforward if the tree is infinite. For example, every $Q$-matroid decomposes into a ray of pieces, each of which is isomorphic to $M(K_4)$, as in \autoref{fig:quircuits2}.

\begin{figure}
\begin{center}
 \includegraphics[width=8cm]{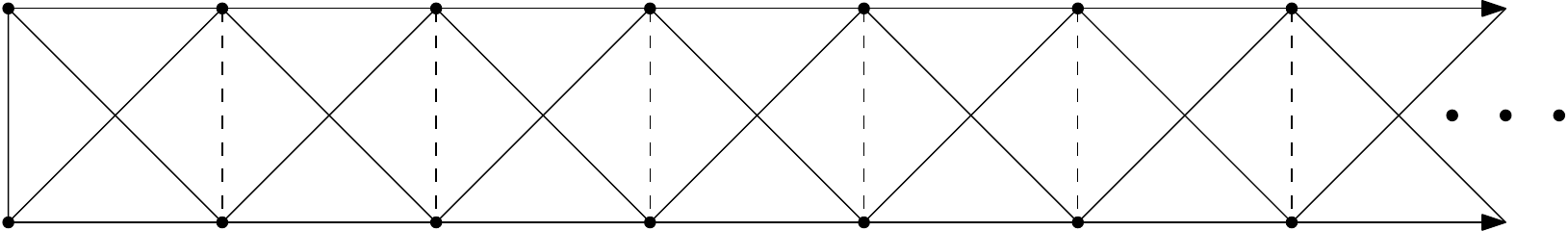}
\end{center} 
\caption{The tree decomposition of the graph $Q$}	\label{fig:quircuits2}
\end{figure}

This example shows that the tree decomposition alone does not provide enough information to reconstruct the matroid: more information is needed. 
In \cite{BCP:quircuits}, we answer the question of which extra information is needed to carry out this reconstruction.
The answer is complicated and beyond the scope of this paper.
However, if we once more restrict our attention to tame matroids then there is a much more natural solution.
In this paper, we give a self-contained account of this more natural solution, for which the necessary arguments are much simpler than in \cite{BCP:quircuits}.

The tree alone is still not enough information - the finite and topological cycle matroids of $Q$ are both tame, and they give rise to the same tree of matroids. Just as in the graphic case, we may think of the topological cycle matroid as the matroid we get by allowing the end of the ray to be used by circuits and the finite cycle matroid as the matroid we get by forbidding circuits to use the end. This suggests what is in fact the right answer: the extra information we need is simply the set $\Psi$ of ends of the tree which may be used by circuits.

More precisely, in \autoref{treesofmatroids} we give a construction which can be thought of as taking infinitely many 2-sums simultaneously. Given a suitable tree $\Tcal$ of matroids and a suitable set $\Psi$ of ends of $\Tcal$, we showed in~\cite{BC:psi_matroids} that this construction allows us to build a matroid $M_{\Psi}(\Tcal)$ by (roughly speaking) sticking together the matroids along the edges of the tree and only allowing circuits to use the ends in $\Psi$. We can show that this construction suffices to rebuild any tame matroid from its canonical decomposition into circuits, cocircuits, and 3-connected pieces, together with information about which ends are used by circuits:

\begin{thm}\label{intrec}
Let $N$ be a tame matroid and let $\Tcal$ be the tree of matroids $\Tcal$ arising from the canonical tree decomposition of $N$.

Then there is some $\Psi\se \Omega(\Tcal)$ such that
$N=M_{\Psi}(\Tcal)$.
\end{thm}

The proof that the $\Psi$-matroids of a locally finite graph really are matroids also relied on gluing together an infinite tree of finite pieces. In this case, the tree structure arose from a tree decomposition of the graph. So both of the results mentioned above say that, for some particular tree structure, if we have any tame matroid whose circuits and cocircuits all fit, in some sense, with that tree structure, then this constrains the matroid to be of a very special type, which we call a $\Psi$-matroid. We give a general result of this type for trees of matroids in which any two adjacent matroids share at most one edge. See \autoref{nice_tree}.

The paper is arranged as follows: we begin by recalling some preliminary results about infinite matroids in general and $\Psi$-matroids in particular in \autoref{prelims}. Then we recall the basic theory of trees of matroids, including the infinitary version of 2-sums, in \autoref{treesofmatroids}. We prove \autoref{intrec} in \autoref{rec} and \autoref{intpsi} in \autoref{sec:psi}. We conclude by discussing open questions and the future outlook in \autoref{outlook}.

\section{Preliminaries}\label{prelims}
Throughout, notation and terminology for (infinite) graphs are those of~\cite{DiestelBook10}, and for matroids those of~\cite{Oxley,matroid_axioms}.

\subsection{Infinite matroids}
$M$ always denotes a matroid and $E(M)$ (or just $E$), $\Ical(M)$ and $\Ccal(M)$ denote its ground 
set and its sets of independent sets and circuits, respectively. 

A set system $\Ical\subseteq \Pcal(E)$ is the set of independent sets of a matroid if and only if it satisfies the following {\em independence  axioms\/} \cite{matroid_axioms}.
\begin{itemize}
	\item[(I1)] $\emptyset\in \Ical(M)$.
	\item[(I2)] $\Ical(M)$ is closed under taking subsets.
	\item[(I3)] Whenever $I,I'\in \Ical(M)$ with $I'$ maximal and $I$ not maximal, there exists an $x\in I'\setminus I$ such that $I+x\in \Ical(M)$.
	\item[(IM)] Whenever $I\subseteq X\subseteq E$ and $I\in\Ical(M)$, the set $\{I'\in\Ical(M)\mid I\subseteq I'\subseteq X\}$ has a maximal element.
\end{itemize}

A set system $\Ccal\subseteq \Pcal(E)$ is the set of circuits of a matroid if and only if it satisfies the following {\em circuit  axioms\/} \cite{matroid_axioms}.
\begin{itemize}
\item[(C1)] $\emptyset\notin\Ccal$.
\item[(C2)] No element of $\Ccal$ is a subset of another.
\item[ (C3)](Circuit elimination) Whenever $X\subseteq o\in \Ccal(M)$ and $\{o_x\mid x\in X\} \subseteq \Ccal(M)$ satisfies $x\in o_y\Leftrightarrow x=y$ for all $x,y\in X$, 
then for every $z \in o\setminus \left( \bigcup_{x \in X} o_x\right)$ there exists a  $o'\in \Ccal(M)$ such that $z\in o'\subseteq \left(o\cup  \bigcup_{x \in X} o_x\right) \setminus X$.

\item[(CM)] $\Ical$ satisfies (IM), where $\Ical$ is the set of those subsets of $E$ not including an element of $\Ccal$.
\end{itemize}


\begin{lem}\label{fdt}
 Let $M$ be a matroid and $s$ be a base.
Let $o_e$ and $b_f$ a fundamental circuit and a fundamental cocircuit with respect to $s$, then
\begin{enumerate}
 \item $o_e\cap b_f$ is empty or $o_e\cap b_f=\{e,f\}$ and
\item $f\in o_e$ if and only if $e\in b_f$.
\end{enumerate}
\end{lem}

\begin{proof}
To see the first note that $o_e\subseteq s+e$ and $b_f\subseteq (E\setminus s)+f$.
So $o_e\cap b_f\subseteq \{e,f\}$. As a circuit and a cocircuit can never meet in only one edge, the assertion follows.

To see the second, first let $f\in o_e$.
Then $f \in o_e \cap b_f$, so by (1) $o_e \cap b_f = \{e, f\}$ and so $e \in b_f$.
The converse implication is the dual statement of the above implication.
\end{proof}

\begin{lem}\label{o_cap_b}
 For any circuit $o$ containing two edges $e$ and $f$, there is a cocircuit $b$ such that $o\cap b=\{e,f\}$.
\end{lem}

\begin{proof}
As $o-e$ is independent, there is a base including $o-e$.
By \autoref{fdt}, the fundamental cocircuit of $f$ of this base intersects $o$ in $e$ and $f$, as desired. \end{proof}

\begin{lem} \label{rest_cir}
 Let $M$ be a matroid with ground set $E = C \dot \cup X \dot \cup D$ and let $o'$ be a circuit of $M' = M / C \backslash D$.
Then there is an $M$-circuit $o$ with $o' \subseteq o \subseteq o' \cup C$.
\end{lem}
\begin{proof}
Let $s$ be any $M$-base of $C$. Then $s \cup o'$ is $M$-dependent since $o'$ is $M'$-dependent.
On the other hand,  $s \cup o'-e$ is $M$-independent whenever $e\in o'$ since $o'-e$ is $M'$-independent.
Putting this together yields that $s \cup o'$ contains an $M$-circuit $o$, and this circuit must not avoid any $e\in o'$, as desired.
\end{proof}

A \emph{scrawl} is a union of circuits. In \cite{BC:rep_matroids},
(infinite) matroids are axiomatised in terms of scrawls. 
The set $\Scal(M)$ denotes the set of scrawls of the matroid $M$.
Dually a \emph{coscrawl} is a union of cocircuits.
Since no circuit and cocircuit can meet in only one element,
no scrawl and coscrawl can meet in only one element. 
In fact, this property gives us a simple characterisation of scrawls in terms of coscrawls and vice versa.

\begin{lem}\label{is_scrawl}\cite{BC:rep_matroids}
Let $M$ be a matroid, and let $w\subseteq E$. The following are equivalent:
\begin{enumerate}
 \item $w$ is a scrawl of $M$.
 \item $w$ never meets a cocircuit of $M$ just once.
 \item $w$ never meets a coscrawl of $M$ just once.
\end{enumerate}
\end{lem}

\begin{proof}
It is clear that (1) implies (3) and (3) implies (2), so it suffices to show that (2) implies (1). 
Suppose that (2) holds and let $e\in w$. Then in the minor $M/(w - e)\setminus(E \setminus w)$ on the groundset $\{e\}$, $e$ cannot be a co-loop, by the dual of \autoref{rest_cir} and (2). So $e$ must be a loop, and by \autoref{rest_cir} there is a circuit $o_e$ with $e \in o_e \subseteq w$. Thus $w$ is the union of the $o_e$, and so is a scrawl.

\end{proof}

\begin{lem}\label{cir_c_scrawl}
 Let $M$ be a matroid and $\Ccal,\Dcal\se \Pcal(E)$
such that every $M$-circuit is a union of elements of $\Ccal$, every $M$-cocircuit is a union of elements of $\Dcal$ and $|C\cap D|\neq 1$ for every $C\in\Ccal$ and every $D\in\Dcal$.

Then $\Ccal(M)\se \Ccal\se \Scal(M)$ and  $\Ccal(M^*)\se \Dcal\se \Scal(M^*)$
\end{lem}

\begin{proof}
We begin by showing that $\Ccal(M) \se \Ccal$. For any circuit $o$ of $M$, pick an element $e$ of $o$. Since $o$ is a union of elements of $\Ccal$ there is $o' \in \Ccal$ with $e \in o' \se o$. Suppose for a contradiction that $o'$ isn't the whole of $o$, so that there is $f \in o \sm o'$. By \autoref{o_cap_b} there is some cocircuit $b$ of $M$ with $o' \cap b = \{e\}$. Then we can find $b' \in \Dcal$ with $e \in b' \se b$, and so $o' \cap b' = \{e\}$, giving the desired contradiction. Similarly we obtain that $\Ccal(M^*) \se \Dcal$.

The fact that $\Ccal \se \Scal(M)$ is immediate from \autoref{is_scrawl} since $\Ccal(M^*) \se \Dcal$, and the proof that $\Dcal \se \Scal(M^*)$ is similar.
\end{proof}

\subsection{$\Psi$-matroids}

In this subsection we shall review the definitions of $\Psi$-circuits and $\Psi^{\complement}$-bonds for a locally finite graph $G$ with a specified set $\Psi$ of ends. Much of what we say will be a review of the early parts of~\cite{DP:dualtrees} and~\cite{BC:psi_matroids}.
We say that two rays in $G$ are {\em equivalent} if they cannot be separated by removing finitely many vertices from $G$. In other words, two rays are equivalent if they may be joined by infinitely many vertex-disjoint paths. An {\em end} of $G$ is an equivalence class of rays under this relation, and the set of ends of $G$ is denoted $\Omega(G)$.

Let $d$ be the distance function on $V(G) \sqcup (0, 1) \times E(G)$ considered as the ground set of the simplicial 1-complex formed from the vertices and edges of $G$. We define a topology on the set $V(G) \sqcup \Omega(G) \sqcup (0,1) \times E(G)$ by taking basic open neighbourhoods as follows:
\begin{itemize}
\itum For $v \in V(G)$, the basic open neighbourhoods of $v$ are the $\epsilon$-balls $B_{\epsilon}(v) = \{x | d(v, x) < \epsilon\}$ for $\epsilon \leq 1$.
\itum For $(x, e) \in (0, 1) \times E$ we say $(x, e)$ is an {\em interior point} of $e$, and take the basic open neighbourhoods to be the $\epsilon$-balls about $(x, e)$ with $\epsilon \leq \min(x, 1-x)$.
\itum For $\omega \in \Omega(G)$, the basic open neighbourhoods of $\omega$ will be parametrised by the finite subsets $S$ of $V(G)$. Given such a subset, we let $C(S, \omega)$ be the unique component of $G - S$ that contains a ray from $\omega$, and let $\hat C(S, \omega)$ be the set of all vertices and inner points of edges contained in or incident with $C(S, \omega)$, and of all ends represented by a ray in $C(S, \omega)$. We take the basic open neighbourhoods of $\omega$ to be the sets $\hat C(S, \omega)$.
\end{itemize}

We call the topological space obtained in this way $|G|$. 

For any set $\Psi$ of ends of $G$, we set $\Psi^{\complement} = \Omega(G) \setminus \Psi$ and $|G|_{\Psi} = |G| \setminus \Psi^{\complement}$. Since $G$ is locally finite, $|G|_{\Psi}$ can be given the structure of a graph-like space in the sense of~\cite{BCC:graphic_matroids} (and closely related to the earlier work of~\cite{ThomassenVellaContinua}): see~\cite{BC:psi_matroids}.

A \emph{$\Psi$-circuit} of $G$ is an edge set whose
$|G|_{\Psi}$-closure is homeomorphic to the unit circle.
A \emph{$\Psi\ct$-bond} of $G$ is a bond of $G$ that has no ends from $\Psi$ in its closure.

We say that \emph{$(G,\Psi)$ induces a matroid $M$} if $E(M)=E(G)$ and the $M$-circuits are the $\Psi$-circuits and the $M$-cocircuits are the $\Psi\ct$-bonds. In this case, we call $M$ the {\em $\Psi$-matroid} $M_{\Psi}(G)$ of $G$. In the following sense the $\Psi$-circuits and $\Psi\ct$-bonds always behave like the circuits and cocircuits of a matroid.

\begin{lem}\label{never_meet_just_once}
 No $\Psi$-circuit meets any $\Psi\ct$-bond in a single edge.
\end{lem}

\begin{proof}
Suppose for a contradiction that some $\Psi$-circuit $o$ meets some $\Psi\ct$-bond $b$ in a single edge $f$

Then $|G|_\Psi$ with all the interior points of edges of $b$ removed has two connected components, namely the two sides of the bond.  
This contradicts the fact that $o-f$ is connected and contains both endvertices of $f$.
\end{proof}

More is shown in~\cite{BC:psi_matroids}. If $\Psi$ is Borel, then the $\Psi$-matroid $M_{\Psi}(G)$ always exists.

\section{Trees of matroids}\label{treesofmatroids}

In this section we review the relationship between trees of matroids and tree decompositions of matroids.

\begin{dfn}
A {\em tree $\Tcal$ of matroids} consists of a tree $T$, together with a function $M$ assigning to each node $t$ of $T$ a matroid $M(t)$ on ground set $E(t)$, such that for any two nodes $t$ and $t'$ of $T$, if $E(t) \cap E(t')$ is nonempty then $tt'$ is an edge of $T$.

For any edge $tt'$ of $T$ we set $E(tt') = E(t) \cap E(t')$. We also define the {\em ground set} of $\Tcal$ to be $E = E(\Tcal) = \left(\bigcup_{t \in V(T)} E(t)\right) \setminus \left(\bigcup_{tt' \in E(T)} E(tt')\right)$. 

We shall refer to the edges which appear in some $E(t)$ but not in $E$ as {\em dummy edges} of $M(t)$: thus the set of such dummy edges is $\bigcup_{tt' \in E(T)} E(tt')$.
\end{dfn}

The idea is that the dummy edges are to be used only to give information about how the matroids are to be pasted together, but they will not be present in the final pasted matroid, which will have ground set $E(\Tcal)$. 

\begin{dfn}
 If $T$ is a tree, and $tu$ is a (directed) edge of $T$, we take $T_{t \to u}$ to be the connected component of $T - t$ that contains $u$. If $\Tcal = (T, M)$ is a tree of matroids, we take $\Tcal_{t \to u}$ to be the tree of matroids $(T_{t \to u}, M \restric_{T_{t \to u}})$.
\end{dfn}

\begin{dfn}
A tree $\Tcal = (T, M)$ of matroids is {\em of overlap 1} if, for every edge $tt'$ of $T$, $|E(tt')| = 1$. In this case, we denote the unique element of $E(tt')$ by $e(tt')$.

Given a tree of matroids of overlap 1 as above, a {\em precircuit} of $\Tcal$ consists of a connected subtree $C$ of $T$ together with a function $o$ assigning to each vertex $t$ of $C$ a circuit of $M(t)$, such that for any vertex $t$ of $C$ and any vertex $t'$ adjacent to $t$ in $T$, $e(tt') \in o(t)$ if and only if $t' \in C$. Given a set $\Psi$ of ends of $T$, such a precircuit is called a {\em $\Psi$-precircuit} if all ends of $C$ are in $\Psi$. The set of $\Psi$-precircuits is denoted $\overline\Ccal(\Tcal, \Psi)$. 

Any $\Psi$-precircuit $(C, o)$ has an {\em underlying set} $\underline{(C, o)} = E \cap \bigcup_{t \in V(C)} o(t)$. Minimal nonempty subsets of $E$ arising in this way are called {\em $\Psi$-circuits} of $\Tcal$. The set of $\Psi$-circuits of $\Tcal$ is denoted $\Ccal(\Tcal, \Psi)$.
\end{dfn}

\begin{dfn}
Let $\Tcal = (T, M)$ be a tree of matroids. Then the {\em dual} $\Tcal^*$ of $\Tcal$ is given by $(T, M^*)$, where $M^*$ is the function sending $t$ to $(M(t))^*$. For a subset $C$ of the ground set, the tree of matroids $\Tcal/C$ obtained from  $\Tcal$ by {\em contracting} $C$ is given by $(T, M/C)$, where $M/C$ is the function sending $t$ to $M(t)/(C \cap E(t))$. For a subset $D$ of the ground set, the tree of matroids $\Tcal\backslash D$ obtained from  $\Tcal$ by {\em deleting} $D$ is given by $(T, M \backslash D)$, where $M \backslash D$ is the function sending $t$ to $M(t) \backslash (C \cap E(t))$. 
\end{dfn}

\begin{lem}
For any tree $\Tcal$ of matroids, $\Tcal = \Tcal^{**}$. For any disjoint subsets $C$ and $D$ of the ground set of $\Tcal$ we have $(\Tcal / C)^* = \Tcal^* \backslash C$, $(\Tcal \backslash D)^* = \Tcal^* / D$ and $\Tcal / C \backslash D = \Tcal \backslash D /C$. If $\Tcal$ has overlap 1 and $(\Tcal, \Psi)$ induces a matroid $M$, then $(\Tcal/C\backslash D, \Psi)$ induces the matroid $M /C \backslash D$ and $(\Tcal^*, \Psi\ct)$ induces the matroid $M^*$. \qed
\end{lem}

We will sometimes use the expression {\em $\Psi\ct$-cocircuits of $\Tcal$} for the $\Psi\ct$-circuits of $\Tcal^*$. If there is a matroid whose circuits are the  $\Psi$-circuits of $\Tcal$ and whose cocircuits are the $\Psi\ct$-cocircuits of $\Tcal$ then we will call that matroid the {\em $\Psi$-matroid} for $\Tcal$, and denote it $M_{\Psi}(\Tcal)$.

\begin{lem}[Lemma 5.5,~\cite{BC:psi_matroids}]\label{2SumsO1}
Let $\Tcal = (T, M)$ be a tree of matroids of overlap 1, $\Psi$ a set of ends of $T$, and let $(C, o)$ and $(D, b)$ be respectively a $\Psi$-precircuit of $\Tcal$ and a $\Psi\ct$-precircuit of $\Tcal^*$. Then $|\underline{(C, o)} \cap \underline{(D, b)}| \neq 1$.
\end{lem}

\begin{thm}[\cite{BC:psi_matroids}]\label{really_matroids}
If $\Tcal = (T, M)$ is a tree of matroids of overlap 1, and $\Psi$ is a Borel set of ends of $T$, then there is a $\Psi$-matroid for $\Tcal$.
\end{thm}

\begin{rem}
In particular, we always get a $\Psi$-matroid when $\Psi$ is the empty set or the set $\Omega(T)$ of all ends of $T$. It is clear that if each matroid $M(t)$ is finitary, then so is $M_\emptyset(\Tcal)$.
\end{rem}

So far we have discussed how to construct matroids by gluing together trees of `smaller' matroids. Now we turn to a notion, taken from~\cite{ADP:decomposition}, of a decomposition of a matroid into a tree of such smaller parts.

\begin{dfn}
A {\em tree decomposition of adhesion 2} of a matroid $N$ consists of a tree $T$ and a partition $R = (R_v)_{v \in V(T)}$ of the ground set $E$ of $N$ such that for any edge $tt'$ of $T$ the partition $(\bigcup_{v \in V(T_{t \to t'})} R_v, \bigcup_{v \in V(T_{t' \to t})} R_v)$ is a 2-separation of $N$.

Given such a tree decomposition, and a vertex $v$ of $T$, we define a matroid $M(v)$, called the {\em torso} of $T$ at $v$, as follows: the ground set of $M(v)$ consists of $R_v$ together with a new edge $e(vv')$ for each edge $vv'$ of $T$ incident with $v$. For any circuit $o$ of $N$ not included in any set $\bigcup_{t \in V(T_{v \to v'})} R_v$, we have a circuit $\hat o(v)$ of $M(v)$ given by $(o \cap R_v )\cup \{e(vv') \in E(v) | o \cap \bigcup_{t \in V(T_{v \to v'})} R_v\neq \emptyset\}$. These are all the circuits of $M(v)$.

In this way we get a tree of matroids $\Tcal(N, T, R) = (T, v \mapsto M(v))$ of overlap 1 from any tree decomposition of adhesion 2. For any circuit $o$ of $N$ we get a corresponding precircuit $(S_o, \hat o)$, where $S_o$ is just the subtree of $T$ consisting of those vertices $v$ for which $\hat o(v)$ is defined.
\end{dfn}

Note that $\underline{(S_o, \hat o)} = o$. Each $M(v)$ really is a matroid~\cite[\S4, \S8]{ADP:decomposition}, isomorphic to a minor of $N$~\cite{BCP:quircuits}, and that $\Tcal(N^*, T, R) = (\Tcal(N, T, R))^*$.
\cite{ADP:decomposition} also contains the following theorem.

\begin{thm}[Aigner-Horev, Diestel, Postle]\label{deco}
 For any matroid $N$ there is a tree decomposition $\Dcal(N)$ of adhesion 2 of $N$ such that all torsos have size at least 3 and are either circuits, cocircuits or 3-connected, and in which no two circuits and no two cocircuits are adjacent in the tree. This decomposition is unique in the sense that any other tree decomposition with these properties must be isomorphic to it.
\end{thm}

The above theorem is a generalisation to infinite matroids of a standard result about finite matroids~\cite{findec1, findec2}. If $N$ is a finite matroid, it is possible to reconstruct $N$ from the decomposition $\Dcal(N)$. However, as noted in the introduction, it is not in general possible to reconstruct $N$ from $\Dcal(N)$ if $N$ is infinite. Our aim in the next section will be to show that if $N$ is tame, then not much extra information is needed to recover $N$. All we need is the set $\Psi$ consisting of those ends of $T$ that appear in the closure of some circuit of $N$.

\section{Reconstruction}\label{rec}

Let $N$ be a tame matroid and let $(T, R)$ be a tree decomposition of $N$ of adhesion 2. We begin by considering the case that $T$ is a ray $t_1, t_2, \ldots$. In this case, we can show that the tree $\Tcal(N, T, R)$ is well behaved.

\begin{dfn}
A precircuit $(S, o)$ for a tree $\Tcal = (T, M)$ of matroids of overlap 1 is called a {\em phantom precircuit} if there is an edge $tt'$ of $S$ such that $o(v) \cap E(\Tcal) = \emptyset$ for $v \in V(S_{t \to t'})$.

$\Tcal=(T,M)$ is \emph{nice} if neither $\Tcal$ nor $\Tcal^*$ has any phantom precircuits.
\end{dfn}

Note that $\Tcal=(T,M)$ is nice iff there is not $tt'\in E(T)$ such that
in $\Tcal_{t \to t'} = (T_{t \to t'}, M \restric_{V(T_{t \to t'})})$ the edge $e(tt')$ is either a loop in $M_{\Omega(\Tcal_{t \to t'})}(\Tcal_{t \to t'})$ 
or a coloop $M_{\emptyset}(\Tcal_{t \to t'})$.

\begin{lem}\label{ray+td_nice}
 Let $N$ be a matroid with a tree decomposition $(T,R)$ of adhesion 2.
\begin{enumerate}
 \item For every $N$-circuit $o$ its corresponding precircuit $(S_o,\hat o)$ is not phantom.
\item If $T$ is a ray, and there is a circuit $o$ and a cocircuit $b$ of $N$ that both have edges in infinitely many of the $R_v$, then $\Tcal(N,T,R)$ is nice.
\end{enumerate}
\end{lem}

\begin{proof}
(1) follows from the definition of $S_o$.

For (2), let $T=t_1,t_2,\ldots$ be a ray. 
Now suppose for a contradiction that 
there is a phantom precircuit $(S_c,c)$. Then for all sufficiently large $n$,
the circuit $c(t_n)$ consists of $e(t_{n-1}t_n)$ and $e(t_{n}t_{n+1})$.
In other words, $e(t_{n-1}t_n)$ and $e(t_{n}t_{n+1})$ are in parallel.

So $c(t_n)\se \hat o(t_n)$, hence $c(t_n)= \hat o(t_n)$.
This contradicts (1).
The case that there is a phantom precocircuit $(S_c,c)$ is similar.
Hence $\Tcal(N,T,R)$ is nice.
\end{proof}

\begin{lem}\label{cir_in_omega_cir}
Let $\Tcal=(T,M)$ be a nice tree of matroids, then every $\emptyset$-circuit is
an $\Omega(T)$-circuit.
\qed
\end{lem}

By duality, an analogue of \autoref{cir_in_omega_cir} is also true for cocircuits.

\begin{lem}\label{cir_in_cir}
Let $\Tcal=(T,M)$ be a nice tree of matroids, and $N$ be a matroid such that 
$\Ccal(N)\se \Ccal(M_{\Omega(T)}(\Tcal))$ and  $\Ccal(N^*)\se \Ccal(M^*_{\emptyset}(\Tcal))$.

Then $\Ccal(M_{\emptyset}(\Tcal))\se \Ccal(N)$ and  $\Ccal(M^*_{\Omega(T)}(\Tcal))\se \Ccal(N^*)$.
\end{lem}

\begin{proof}
 By duality, it suffices to prove only that $\Ccal(M_\emptyset(T))\se \Ccal(N)$.
So let $o\in \Ccal(M_\emptyset(T))$. 
Since $o$ never meets an element of $\Ccal(M^*_{\emptyset}(T))$ just once,
it never meets an element of $\Ccal(N^*)$ just once.
Hence $o$ includes an $N$-circuit $o'$ by \autoref{is_scrawl}.
Thus $o'\in  \Ccal(M_{\Omega(T)}(T))$. By \autoref{cir_in_omega_cir},
we must have $o'=o$. So $o\in \Ccal(N)$, as desired.
\end{proof}

\begin{lem}\label{ray_case}
  Let $N$ be a tame matroid with a tree decomposition $(T,R)$ of adhesion 2.
Assume that $T=t_1t_2\ldots$ is a ray.

Then there are not a circuit $o$ and a cocircuit $b$ of $N$ that both converge to the end $\omega$ of $T$.

Indeed, either $N=M_\emptyset(\Tcal(N,T,R))$ or $N=M_{\{\omega\}}(\Tcal(N,T,R))$

\end{lem}

\begin{proof}
 Suppose for a contradiction that there are such $o$ and $b$. 
Then there are $l<m<n$ and $e_l,e_m,e_n\in E(N)$ such that 
$e_l\in b\cap E(t_{l})$, and $e_m\in o\cap E(t_{m})$, and
$e_n\in b\cap E(t_{n})$. 
Using the tameness of $N$, we make these choices in such a way that for any $i \geq m$, the intersection of $o \cap b$ with $E(t_i)$ is empty.
We may also assume that $o$ has an edge in some $E(t_k)$ with $k<m$, so that both dummy edges of $E(t_m)$ are in $\hat o(t_m)$.

Now $b$ is a cocircuit of $M(\Tcal(N,T,R),\emptyset)$ since $(S_b,\hat b)$ is a precocircuit,
and there cannot be a precocircuit 
whose cocircuit at any node $t$ is a subset of $\hat b(t)$. 
By the dual of \autoref{o_cap_b} there is some $M_{\emptyset}(\Tcal)$-circuit $o_b$ meeting $b$ only in $e_l$ and $e_n$. 
Note that $o_b$ is also a circuit of 
$N$ by \autoref{cir_in_cir} since $\Tcal(N,T,R)$ is nice by \autoref{ray+td_nice}.

Now we build an $\emptyset$-precircuit $(S_C, \hat C)$ as follows. First we set 
$S_C =(S_{o_b}\sm \{t_1\ldots t_m\}) \cup (S_{o}\cap \{t_1\ldots t_m\})$. We take 
$\hat C(t_j)=\hat o_b(t_j)$ for $j> m$, and $\hat C(t_j)=\hat o(t_j)$ for $j\leq m$.
Let $C$ be the underlying circuit of $(S_C, \hat C)$. Note that $C$ is a circuit of 
$M_{\emptyset}(\Tcal)$ and so also a circuit of $N$ by \autoref{cir_in_cir} and \autoref{ray+td_nice} as before.

We now apply circuit elimination in $N$ to the circuits $o$ and $C$, eliminating the edge $e_m$ and keeping the edge $e_n$. Call the resulting circuit $C'$.

If $t_m\in S_{C'}$, then $\hat C'(t_m)\subseteq \hat o(t_m)-e_m$ (since both dummy edges of $E(t_m)$ are in $\hat o(t_m)$), which is impossible.
So $S_{C'}\se \{t_{m+1},t_{m+2},\ldots \}$. Hence $C'\cap b=\{e_n\}$, which is also impossible.

We have now established that there cannot be a circuit $o$ and a cocircuit $b$ of $N$ such that $\omega$ is in the closure of both $o$ and $b$. 

If $\omega$ is in the closure of some $N$-circuit, then every $N$-circuit is a $\{\omega\}$-circuit, and every $N$-cocircuit is an $\emptyset$-cocircuit. Since by \autoref{2SumsO1} no $\Psi$-circuit ever meets a $\Psi\ct$-cocircuit just once, we may apply \autoref{cir_c_scrawl} to deduce that $N = M_{\{\omega\}}(\Tcal)$.
In the case that $\omega$ is not in the closure of any $N$-circuit a similar argument yields that
$N = M_{\emptyset}(\Tcal)$.
This completes the proof.
\end{proof}

Having considered the case that the tree $T$ is a ray, we now reduce the general case to this special case.
Let $N$ be a matroid with a tree decomposition $(T,R)$ of adhesion 2.
Let $Q=q_1,q_2,\ldots $ be a ray in $T$. We define $R^Q$ to be the following coarsening 
of $R$. We define $R^Q_{q_i}$ to be the union of all the $R_{v}$ such that in $T$ the vertices
$v$ and $q_i$ can be joined by a path that does not contain any other $q_j$.

Then $(Q,R^Q)$ is a tree decomposition of $N$ of adhesion 2.
An $N$-circuit $o$ has the end $\omega$ of $Q$ in its closure with respect to $(T,R)$ if and only if $o$ has $\omega$ in its closure with respect to $(Q,R^Q)$.
So by \autoref{ray_case}, we deduce that there cannot be a circuit and a cocircuit of 
$N$ that have a common end in both of their closures (with respect to $(T,R)$).

Let $\Psi$ be the set of ends of $T$ that appear in the closure of some circuit of $N$. Thus every $N$-circuit is a $\Psi$-circuit and every $N$-cocircuit is a $\Psi\ct$-cocircuit. Since by \autoref{2SumsO1} no $\Psi$-circuit ever meets a $\Psi\ct$-cocircuit just once, we may apply \autoref{cir_c_scrawl} to deduce that $N = M_{\Psi}(\Tcal)$. Hence we get the following theorem.

\begin{thm}
Let $N$ be a tame matroid with a tree decomposition $(T,R)$ of adhesion 2.

Then there is some $\Psi\se \Omega(T)$ such that
$N=M(\Tcal(N,T,R),\Psi)$.
\end{thm}

Combining this theorem with \autoref{deco} yields:

\begin{thm}
Let $N$ be a connected tame matroid.  
Then $N = M_{\Psi}(\Tcal)$ where each $M(t)$ is either a circuit, a cocircuit or else is 3-connected.
\end{thm}

\begin{rem}
 In the proof of this theorem, it might look as if we would have some freedom in choosing the set $\Psi$, namely that we could take $\Psi$ to be any set containing all the ends to which some circuit converges and avoiding all ends to which some cocircuit converges. 
However, it can be shown that for every end in $\Psi$, there is a $\Psi$-circuit having this end in the closure.
\end{rem}

The arguments above make use of the fact that the matroid $N$ is severly constrained by the restriction that each of its circuits comes from some precircuit of the tree, and each of its cocircuits comes from some precocircuit. In investigating how restrictive constraints of this form might be in general, we are led to the following question. Suppose that we have a tree $\Tcal = (T, M)$ of matroids. We say a matroid $N$ is a $\Tcal$-matroid if every circuit of $N$ is an $\Omega(T)$-circuit and every cocircuit of $N$ is an $\emptyset$-cocircuit. How constrained is $N$? If $\Tcal$ is not nice, then $N$ can be quite unconstrained.

\begin{eg}
Here the tree $T$ is a ray and each $M(t)=M(C_4)$, arranged as in \autoref{fig:nice_needed}.
Then $M_\emptyset(\Tcal)$ is the free matroid but $M_{\Omega(T)}(\Tcal)$
consists of a single infinite circuit. So any pair of edges forms an $M_{\Omega(T)}(\Tcal)$-cocircuit which is not an $M_\emptyset(\Tcal)$-cocircuit.

 \begin{figure}
\begin{center}
 \includegraphics[width=8cm]{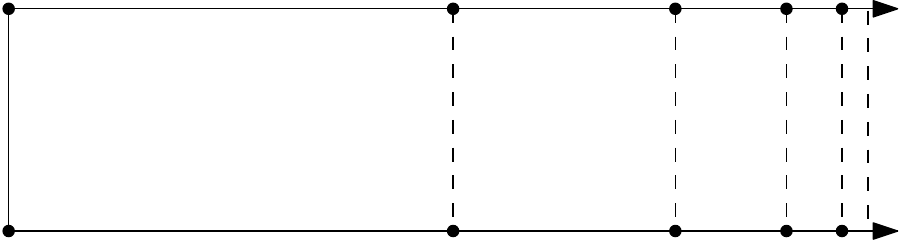}
\end{center} 
\caption{A non-nice tree of matroids}\label{fig:nice_needed}
\end{figure}
\end{eg}

However, if $\Tcal$ is nice and $N$ is tame then $N$ has to be of the form $M_{\Psi}(\Tcal)$:

\begin{thm}\label{nice_tree}
 Let $\Tcal = (T, M)$ be a nice tree of matroids of overlap 1, and let $N$ be a tame $\Tcal$-matroid. Then there is some $\Psi \se \Omega(T)$ such that $N = M_{\Psi}(\Tcal)$.
\end{thm}

 \begin{figure}
\begin{center}
 \includegraphics[width=6cm]{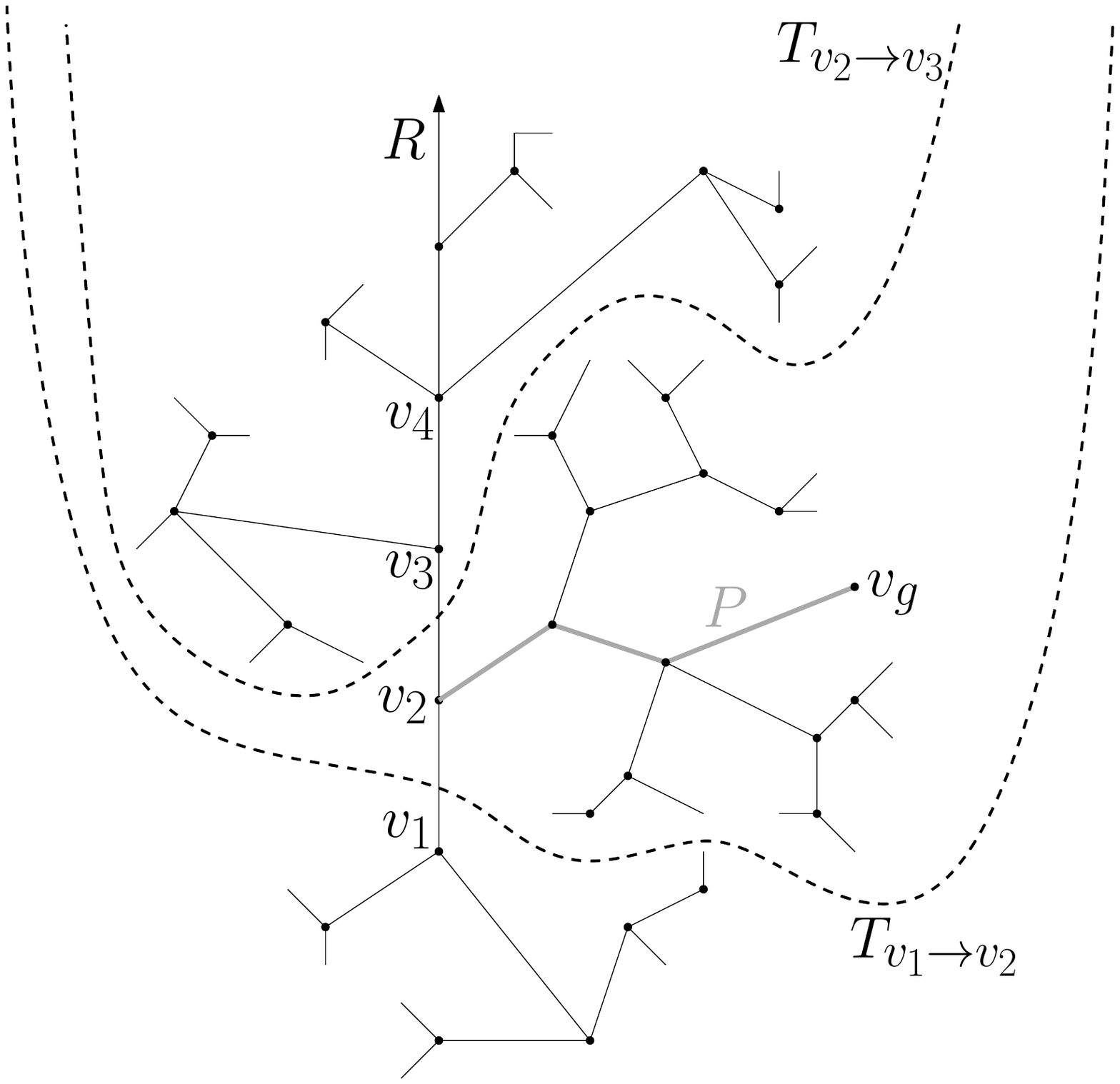}
\end{center} 
\caption{Objects appearing in the proof of \autoref{nice_tree}}	\label{fig:tree}
\end{figure}
\begin{proof}
 We begin by showing that there cannot be a circuit $o$ and a cocircuit $b$ of $N$ such that there is some end $\omega$ of $T$ in the closure of both $o$ and $b$. So suppose for a contradiction that there are such $o$, $b$, and $\omega$. We fix some notation, as illustrated in \autoref{fig:tree}. Pick a ray $R = v_1, v_2, \ldots$ in $T$ to $\omega$. By taking a suitable tail of $R$ if necessary, we may assume that there is some edge $f$ of $b$ in $E(\Tcal) \sm E(\Tcal_{v_1 \to v_2})$, some edge $g$ of $o$ in $E(\Tcal_{v_1 \to v_2}) \sm E(\Tcal_{v_2 \to v_3})$ and some edge $h$ of $b$ in $E(\Tcal_{v_2 \to v_3})$ (here we use that $\Tcal$ is nice). Since $o \cap b$ is finite, we may even assume that no edge of $o \cap b$ lies in $E(\Tcal_{v_1 \to v_2})$.

We may also assume that $o$ has an edge in $E(\Tcal_{v_2 \to v_1})$, so that both dummy edges of $E(v_2)$ are in $\hat o(v_2)$.

By \autoref{o_cap_b} there is some $M_{\emptyset}(\Tcal)$-circuit $o_b$ meeting $b$ only in $f$ and $h$. Let $(S, \hat o)$ be an $\Omega(T)$-precircuit representing $o$, and $(S_b, \hat o_b)$ be an $\emptyset$-precircuit representing $o_b$. Let $v_g$ be the node of $T$ with $g \in E(v)$. Let $P$ be the path joining $v_2$ to $v_g$ in $T$. Let $\partial$ be the set of edges $tt'$ of $T$ with $t$ in $V(P)$ but $t'$ not in either $V(P)$ or $V(T_{v_2 \to v_3})$. For each edge $tt' \in \partial$ there is by niceness of $\Tcal$ some $M_{\emptyset}(\Tcal_{t \to t'})$-circuit $o_{t \to t'}$ through $e(tt')$. Let $(S_{t \to t'}, \hat o_{t \to t'})$ be an $M_{\emptyset}(\Tcal_{t \to t'})$-precircuit representing $o_{t \to t'}$. 

Now we build a $\emptyset$-precircuit $(S_C, \hat C)$ from all this data as follows. First we set $$S_C = (S_b \cap T_{v_2 \to v_3}) \cup P \cup \bigcup_{tt' \in \partial} S_{t \to t'}.$$
Then we take $\hat C(u)$ to be $\hat o_b(u)$ for $u \in (S_b \cap T_{v_2 \to v_3})$, $\hat o(u)$ for $u \in P$ and $\hat o_{t \to t'}(u)$ for $u \in S_{t \to t'}$. Let $C$ be the underlying circuit of $(S_C, \hat C)$. By \autoref{cir_in_cir}, $C$ is an $N$-circuit.

We now apply circuit elimination in $N$ to the circuits $o$ and $C$, eliminating the edge $g$ and keeping the edge $h$. Call the resulting circuit $C'$, and let $(S_{C'}, \hat C')$ be an $\Omega(T)$-precircuit representing $C'$. Let the vertices of $P$ be, in order, $v_g = p_1, p_2, \ldots p_k = v_2$. We shall show by induction on $i$ that $p_i \not \in S_{C'}$. For the base case, we note that if $v_g$ were in $S_{C'}$ we would have to have $\hat C' (v_g) \se \hat o(v_g) \sm \{g\}$, which is impossible. For the induction step, we similarly note that if $p_{i+1}$ were in $S_{C'}$ we would have to have $\hat C' (p_{i+1}) \se \hat o(p_{i+1}) \sm \{e(p_ip_{i+1})\}$, by the induction hypothesis, which is impossible. In particular, we deduce that $v_2 \not \in S_{C'}$. On the other hand, we know that $h \in C'$ so that $S_{C'} \se T_{v_2 \to v_3}$, so that $C' \cap b = \{h\}$, a contradiction.

We have now established that there cannot be a circuit $o$ and a cocircuit $b$ of $N$ such that there is some end $\omega$ of $T$ in the closure of both $o$ and $b$. Let $\Psi$ be the set of ends of $T$ that appear in the closure of some circuit of $N$. Thus every $N$-circuit is a $\Psi$-circuit and every $N$-cocircuit is a $\Psi\ct$-cocircuit. Since by \autoref{2SumsO1} no $\Psi$-circuit ever meets a $\Psi\ct$-cocircuit just once, we may apply \autoref{cir_c_scrawl} to deduce that $N = M_{\Psi}(\Tcal)$ as required.
\end{proof}

\section{Tame $G$-matroids are $\Psi$-matroids}\label{sec:psi}

Let $G$ be a locally finite graph. Recall that a matroid $N$ on the ground set $E(G)$ is a $G$-matroid if $\Ccal(N)\se \Ccal(M_C(G))$ and  $\Ccal(N^*)\se \Ccal(M_{FC}(G)^*)$. Since $\Ccal(M_{FC}(G)) \se \Ccal(M_C(G))$ and $\Ccal(M_C^*(G)) \se \Ccal(M_{FC}^*(G))$ both $M_{FC}(G)$ and $M_C(G)$ are $G$-matroids, and an argument like that for \autoref{cir_in_cir} shows that for any $G$-matroid $N$ we have $\Ccal(M_{FC}(G)) \se \Ccal(N)$ and $\Ccal(M_C^*(G)) \se \Ccal(N^*)$. The aim of this section is to prove the following.

\begin{thm}\label{cycle_matroids}
Let $G$ be a locally finite graph, and let $N$ be a tame $G$-matroid.
There there is some $\Psi\se\Omega(G)$ such that $N=M_{\Psi}(G)$.
\end{thm}

For the rest of this section we fix some locally finite graph $G$ and some tame $G$-matroid $N$.

In this section, we will have to use two different notions of path. 
Finite paths in graphs will simply be called \emph{paths}, whereas 
paths in the topological sense, namely continuous images of the closed unit interval,
will be called \emph{topological paths}.

For a pair of points on a topological circle, there are two arcs joining them through the circle. To allow us to distinguish them, we shall make use of orientations of circles and topological paths. For distinct points $x$ and $y$ on an oriented circle $\vec o$ we use $x \vec o y$ to denote the (oriented) topological path from $x$ to $y$ through $o$ whose orientation agrees with that of $\vec o$. We denote the other topological path by $x \cev o y$. If $x = y$, we do not take the trivial topological path but the topological path that goes all the way around the circle.

 \begin{figure}
\begin{center}
 \includegraphics[width=4cm]{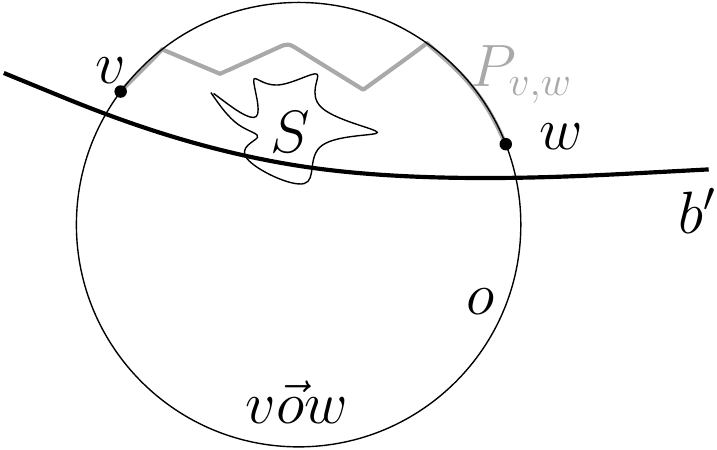}
\end{center} 
\caption{The situation of \autoref{shortcut_exists}.}	\label{fig:shortcut_exists}
\end{figure}

\begin{lem}\label{shortcut_exists}
Let $o$ be a topological circle in $G$.
Let $v,w\in V(o)$ and let $S$ be a finite set of vertices avoiding $V(o)$.
Then for any orientation $\vec{o}$ of $o$, there is a finite $v$-$w$-path $P_{v,w}$ not meeting $v\vec{o}w$ in interior points and avoiding $S$.

Moreover 
if $v\vec{o}w$ has at least two edges, then
there is a bond $b'$ of $G$ that 
has  $P_{v,w}\cup v\cev{o}w$ on one side and all interior vertices of $v\vec{o}w$ on the other side.
\end{lem}

\autoref{fig:shortcut_exists} gives an overview of the terminology used in this lemma.

\begin{proof}
The proof is trivial if $v=w$. Thus we may assume that $v\neq w$.
Let $e_v$ be the first edge on $v\vec{o}w$ and $e_w$ be the last edge on $v\vec{o}w$.
Note that $e_v$ and $e_w$ exist since $v$ and $w$ are vertices. If $e_v=e_w$, we pick
$P_{v,w}=e_v$. So we may assume that $e_v \neq e_w$.

Let $G'=G\sm S$. Since $o$ is a topological circle in $|G'|$, 
there is a finite bond $b$ of $G'$ meeting $o$ in precisely $e_v$ and $e_w$.

All edges and vertices of $v\vec{o}w-e_v-e_w-v-w$ are on the same side of $b$.
Let $C$ be the other side. Note that $v$ and $w$ are in $C$. Now let $P_{v,w}$ 
be some path in $C$ joining $v$ and $w$.

The bond $b$ extends to a finite bond $b'$ of $G$ by adding finitely many deleted edges. 
The bond $b'$ has the desired property.
\end{proof}

 \begin{figure}
\begin{center}
 \includegraphics[width=4cm]{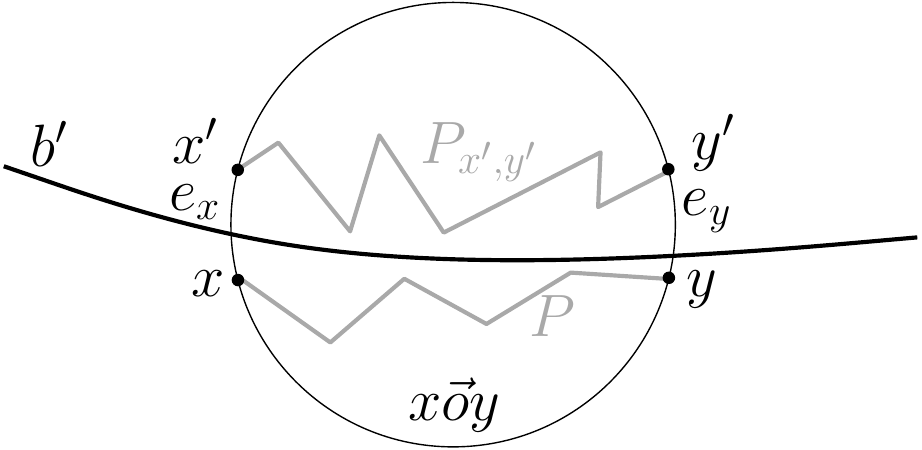}
\end{center} 
\caption{The situation of \autoref{shortcut_can_be_used}.}	\label{fig:shortcut_can_be_used}
\end{figure}

\begin{lem}\label{shortcut_can_be_used}
Let $o$ be a circuit of the $G$-matroid $N$, and $\vec{o}$ be some orientation of $o$.
Further, let $x,y\in V(o)$ and let $\vec{P}$ be an 
$x$-$y$-path meeting $o$ in precisely $x$ and $y$.

Then $x\vec{o}y\cev{P}x\in \Ccal(N)$.
\end{lem}

\begin{proof}
The proof is trivial if $x=y$. Thus we may assume that $x\neq y$.
Let $e_x$ be the first edge on $x\cev{o}y$ and $e_y$ be the last edge on $x\cev{o}y$.
Let $x'$ be the endvertex of $e_x$ that is not $x$, and $y'$ be the endvertex of $e_y$ that 
is not $y$, as depicted in \autoref{fig:shortcut_can_be_used}.

Applying \autoref{shortcut_exists} to $\vec{o}$ with $S=V(P)-x-y$, yields
an $x'$-$y'$-path $P_{x',y'}$, and a bond $b'$ as in that lemma.
By assumption the finite circuit $x\vec{P}ye_yy'\cev{P_{x',y'}}x'e_xx$ is an $N$-circuit.
See \autoref{fig:shortcut_can_be_used} to get an overview of all the definitions.

Now we apply circuit elimination in $N$ to this new circuit and $o$ eliminating $e_x$ and keeping some $z\in E(P)$. Note that $z$ exists since $x\neq y$.
We obtain an $N$-circuit $o'\se (o-e_x)\cup P\cup P_{x',y'}$ including $z$.

It remains to show that $o'= x\vec{o}y\cup P$.
Since each vertex of $G$ is incident with $0$ or $2$ edges of $o'$, we conclude that 
each edge adjacent to $z$ on $P$ is in $o'$. In fact an inductive argument yields that
$P\se o'$. 

If $x\cev{o}y$ consists of a singe edge $xy$, then by the same argument $xy$ also cannot be in $o'$.
Thus $o'\se x\vec{o}y\cev{P}x$. Since the latter is a topological cycle, we must have equality, hence $x\vec{o}y\cev{P}x \in \Ccal(N)$.
Thus we may assume that $x\cev{o}y$ includes at least two edges.

We know that $o'$ is the union of $P$ and some topological arc $A$ from $x$ to $y$. 
The edge set of this arc is included in $L:=(o-e_x)\cup P_{x',y'}$. 
The set $L$ meets $b'$ precisely in $e_y$. 
Let $K$ be the side of $b'$ not containing $x$ and $y$.
Suppose for a contradiction that $L$ includes an edge $e_k$ from $K$.
Then there are two disjoint arcs $L_x$ and $L_y$ from $e_k$ to $x$ and from $e_k$ to $y$.
By the Jumping-Arc Lemma~\cite[Lemma 8.5.3]{DiestelBook10}, both of these have to meet $b'$,
contradicting the fact that $L$ meets $b'$ just in $e_y$. 

This means that $L\se x\vec{o}y$. Since $x\vec{o}y$ is an $x$-$y$-arc, we actually get $L=x\vec{o}y$.
Thus we have shown that $o'= x\vec{o}y\cup P$, which completes the proof.
\end{proof}

\begin{cor}\label{shortcut_can_be_used_reloaded}
Let $o\in \Ccal(N)$, and $\vec{o}$ be some orientation of $o$.
Further, let $x,y\in V(o)$ and let $\vec{P}$ be an 
$x$-$y$-path meeting $x\vec{o}y$ not in interior points.

Then $x\vec{o}y\cev{P}x\in \Ccal(N)$.
\end{cor}

\begin{rem}
 The only difference between \autoref{shortcut_can_be_used} and 
\autoref{shortcut_can_be_used_reloaded} is that in the second the path $P$ may meet 
$o$ in some of the interior points.
\end{rem}

\begin{proof}
We prove this by strong induction on $|P\cap o|$. 
Let $z$ be the second point in the order of $P$ in $P\cap o$ (the first such point is $x$). 
Now we apply \autoref{shortcut_can_be_used} to $o$ and $xPz$ and obtain a new circuit 
$o_z:=x\vec{o}z\cev{P}x$ and a new path $P_z:=zPy$.
Since $|o_z\cap P_z|<|o\cap P|$, we may apply the induction hypothesis.
\end{proof}


\begin{lem}\label{make_1_ended}
Let $o\in \Ccal(N)$ and let $\omega$ be an end of $o$. Then there is $o'\in \Ccal(N)$
that has only the end $\omega$ in its closure.
\end{lem}

\begin{proof}
 First we pick an orientation $\vec{o}$  of $o$.
Then we pick a $\Zbb$-indexed family of distinct edges $e_{i}$
 such that their ordering on $\omega\vec{o}\omega$ is the same as the ordering of their indices and $(e_{i}|i>0)$ and $(e_{i}|i<0)$ both converge to $\omega$. 

Let $s_i$ and $t_i$ be the endvertices of $e_i$ such that $s_i<t_i$ on $\omega\vec{o}\omega$.
We repeatedly apply \autoref{shortcut_can_be_used_reloaded} to get a $\Zbb$-indexed family of vertex-disjoint $t_i$-$s_{i+1}$-paths $P_i$ with $P_i$ disjoint from $t_{i+1}\vec{o}s_{i}$.

Let $D$ be the double ray obtained from sticking the $P_i$ and the $e_i$ together, formally:

\[
 \ldots t_{-1}P_{-1}s_{0}e_{0}t_{0}P_{0}s_{1}e_1t_1P_1 \ldots
\]

By construction, both tails of $D$ belong to $\omega$. So $D$ is a topological cycle.
It remains to show that $D\in\Ccal(N)$.
Suppose not, for a contradiction: Then $D\in \Ical(N)$, so there is a $N$-bond
$b$ with $b\cap D=\{e_0\}$. 

Since $N$ is tame, $o\cap b$ is finite, so there are only finitely many $i\in \Zbb$ 
such that $b$ meets $t_i\vec{o}s_{i+1}$. Let $K$ be the set of such $i$.

By applying \autoref{shortcut_can_be_used_reloaded} finitely often, we get a circuit
$o''$ that meets $b$ precisely in $e_0$. Formally, 
\[
 o''=o\sm \left(\bigcup_{i\in K}(t_i\vec{o}s_{i+1})\right)\cup \left(\bigcup_{i\in K}P_i\right)
\]
So there are a circuit and a cocircuit of $N$ which meet just once, which is the desired contradiction.
\end{proof}

\begin{lem}\label{o_cap_b=0}
Let $b\in \Ccal(N^*)$ and $\omega$ be an end in the closure of $b$. 
Assume there is a double ray $o\in \Ccal(N)$  
both of whose tails converge to $\omega$.

Then there is such an $o$ that does not meet $b$.
 \end{lem}

\begin{proof}
We prove this by induction on $|o\cap b|$; the base case $|o\cap b|=0$ is clear. 
The case $|o\cap b|= 1$ is impossible.
So suppose for the induction step that $|o\cap b|\geq 2$. Thus we may pick $e,f\in o\cap b$.
Since $b\in\Ccal(M_{FC}(G))$, there is a finite circuit $o'$ meeting $b$ in precisely $e$ and $f$. Note that $o'$ is an $N$-circuit. Now pick $z$ in the infinite component of $o\sm o'$ containing $\omega$.

Applying circuit elimination to $o$ and $o'$ eliminating $e$ 
and keeping $z$ yields an $N$-circuit $o''\se o\cup o'-e$ through $z$.
By the choice of $z$, the subgraph with edge set $o\cup o'-e-z$ has two components one of which is a ray $R$ from one endvertex of $z$ converging to $\omega$. Since each vertex is incident with $0$ or $2$ edges of $o''$ and $z\in o''$, the ray $R$ must be included in $o''$.
Hence $o''$ must be infinite, it also has only the end $\omega$ in its closure since 
$o''\se o\cup o'-e$ has no other end in its closure. 
Now $o''\cap b\se (o\cap b)-e$. This completes the induction step.
\end{proof}

\begin{lem}\label{no_end_in_closure}
Let $b\in \Ccal(N^*)$ and $\omega$ be an end in the closure of $b$. 

Then there is no double ray $o\in \Ccal(N)$ both of whose tails converge to $\omega$ with $o\cap b=\emptyset$.
\end{lem}

\begin{proof}
Suppose for a contradiction that there is such an $N$-circuit $o$. 
Let $C_1$ and $C_2$ be the two sides of $b$. Since $o\cap b=\emptyset$ and 
since the double ray $o$ is connected as a subgraph, it lies entirely on one side, say $C_1$.
Since $G$ is locally finite, $C_2$ includes a ray $R$ converging to $\omega$.

 \begin{figure}
\begin{center}
 \includegraphics[width=4cm]{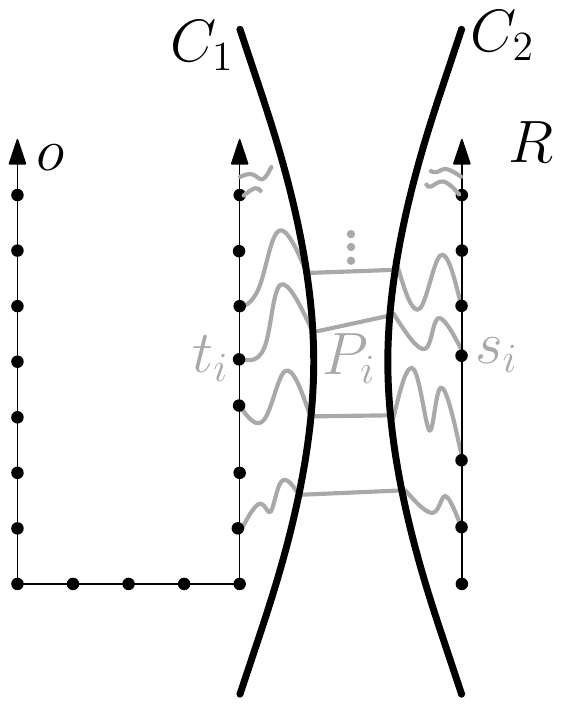}
\end{center} 
\caption{The $o$-$R$-paths $P_i$.}	\label{fig:infinite_eli}
\end{figure}

Now we construct $o$-$R$-paths $P_i$ as in \autoref{fig:infinite_eli}.
Since $o$ and $R$ both have $\omega$ in their closure, there are infinitely many vertex disjoint $R$-$o$-paths $(P_i|i\in \Nbb)$. We enumerate the $P_i$ such that in the linear order on $R$, the starting vertex of $P_i$ is less than the starting vertex of $P_j$ if and only if $i>j$. By Ramsey's theorem there is a tail $R_o$ of $o$ and $N\se \Nbb$ such that all $P_i$
with $i\in \Nbb$ have their endvertex on $R_o$, and in the linear order on $R_o$, the endvertex of $P_i$ is less than the endvertex of $P_j$ if and only if $i>j$.
By relabeling the indices of the $P_i$ if necessary, we may assume that $N=\Nbb$.
Let $s_i$ be the starting vertex of $P_i$, and $t_i$ be its endvertex.

Now we prepare to apply the infinite circuit elimination axiom. We pick some edge $x_i$ between $t_{2i-1}$ and $t_{2i}$ on $R_o$, and pick some $z\in o\sm R_o$.
Then $C_{x_i}=t_{2i-1}R_ot_{2i}P_{2i}s_iRs_{2i-1}P_{2i-1}t_{2i-1}$ is a finite circuit.
So $C_{x_i}$ is an $N$-circuit. We apply circuit elimination 
to $o$ and the $C_{x_i}$ eliminating the $x_i$ and keeping $z$.
Thus there is an $N$-circuit $o'$ through $z$ that is included in:
\[
 \left(o\cup \left(\bigcup_{i\in \Nbb} t_{2i}P_{2i}s_iRs_{2i-1}P_{2i-1}t_{2i-1}\right)\right)\sm \{x_i|i\in\Nbb\}
\]

Since each vertex has degree $0$ or $2$ on $o'$, no edge from any of the finite paths 
$X_i:=t_{2i-1}R_ot_{2i}$ is in $o'$. 
Hence $o'$ is included in:

\[
D:=(o\sm R_o)\cup \left(\bigcup_{i\in \Nbb} C_{x_i}\sm X_i\right)
\]

But $D$ is a double ray. So $D=o'$ and is an $N$-circuit. But $D\cap b$ is infinite. This contradicts the tameness of $N$.
\end{proof}

\begin{proof}[Proof of \autoref{cycle_matroids}.]
 First we show that there cannot be an $N$-circuit $o$, and an $N$-cocircuit $b$ that have a common end in their closure. Suppose for a contradiction there are such $o$ and $b$. By \autoref{make_1_ended}, we get that there is such an $o$ with only the end $\omega$ in its closure. By \autoref{o_cap_b=0}, we get there is such an $o$ that additionally does not meet $b$. By \autoref{no_end_in_closure}, we then get the desired contradiction.
So no end $\omega$ is ever in the closure of both a $N$-circuit and a $N$-cocircuit.

This motivates the following definition. Let $\Psi$ be the set of ends that are in the closure of some $N$-circuit. Then every $N$-circuit is a $\Psi$-circuit and every $N$-cocircuit is a $\Psi\ct$-cocircuit. Since $N$ is a matroid, and the intersection of any $\Psi$-circuit with any $\Psi\ct$-cocircuit is never of size $1$ by \autoref{never_meet_just_once}, we are in a position to apply \autoref{cir_c_scrawl}. Hence $N=M(G,\Psi)$.
\end{proof}

\section{Outlook}\label{outlook}

\autoref{cycle_matroids} and \autoref{nice_tree} are closely related. Consider first \autoref{nice_tree}, applied to a tree $\Tcal = (T, M)$ of finite matroids. Then the matroid $M_{\emptyset}(\Tcal)$ is finitary and the matroid $M_{\Omega(T)}(\Tcal)$ is cofinitary. The niceness condition tells us even more: $M_{\emptyset}(\Tcal)$ is the finitarisation of $M_{\Omega(T)}(\Tcal)$, that is, its circuits are precisely the finite circuits of $M_{\Omega(T)}(\Tcal)$. Dually, $M_{\Omega(T)}(\Tcal)$ is the cofinitarisation of $M_{\emptyset}(\Tcal)$, that is, its cocircuits are precisely the finite cocircuits of $M_{\emptyset}(\Tcal)$. \autoref{nice_tree} determines the lattice of tame matroids $N$ lying between these two in the sense that all their circuits are circuits of $M_{\Omega(T)}(\Tcal)$ and all their cocircuits are cocircuits of $M_{\emptyset}(\Tcal)$.

The situation in \autoref{cycle_matroids} is similar. Since $G$ is a locally finite graph, $M_{FC}(G)$ is the finitarisation of $M_C(G)$ and $M_C(G)$ is the cofinitarisation of $M_{FC}(G)$. Once more, we have characterised the lattice of tame matroids lying between these two. The similarity runs even deeper. In~\cite{BC:psi_matroids} we show that $\Psi$-matroids for a locally finite graph $G$ are naturally thought of as being constructed from trees of finite matroids associated to $G$ (though these trees of matroids need not have overlap 1).

This leads us to consider the following general context. We denote the finitarisation of a matroid $M$ by $M\fin$ and the cofinitarisation of $M$ by $M\cofin$. We say that matroids $M_f$ and $M_c$ are {\em twinned} if $M_f = M_c\fin$ and $M_c = M_f\cofin$. Then we say that $N$ {\em lies between} $M_f$ and $M_c$ if all of its circuits are $M_c$-circuits and all of its cocircuits are $M_f$-cocircuits. 

\begin{oque}
 What is the structure of the lattice of tame matroids between $M_f$ and $M_c$?
\end{oque}

Examples of twinned pairs $(M_f, M_c)$ of matroids abound.

\begin{prop}
Let $M$ be any finitary matroid. Then $((M\cofin)\fin)\cofin = M\cofin$, so that $(M\cofin)\fin$ and $M\cofin$ are twinned.
\end{prop}
\begin{proof}
No circuit of $(M\cofin)\fin$ ever meets a cocircuit of $M\cofin$ in just one element, so every cocircuit $b$ of $M\cofin$ is a coscrawl of $(M\cofin)\fin$ by the dual of \autoref{is_scrawl}. The cocircuits in the union must all be cocircuits of $((M\cofin)\fin)\cofin$ since $b$ is finite. Similarly, every circuit of $M$ is a scrawl of $(M\cofin)\fin$, so that no cocircuit of $((M\cofin)\fin)\cofin$ ever meets a circuit of $M$ in just one element, and so every cocircuit $b$ of $((M\cofin)\fin)\cofin$ is a coscrawl of $M\cofin$ by the dual of \autoref{is_scrawl} applied to $M$ and finiteness of $b$. Thus $((M\cofin)\fin)\cofin$ and $M\cofin$ have the same cocircuits.
\end{proof}

Thus for example if we begin with the finite cycle matroid of a graph $G$ then its cofinitarisation is twinned with the finitarisation of its cofinitarisation. In fact these are just the topological cycle matroid and finite cycle matroid of the finitely separable quotient of $G$, in which we identify any two vertices that cannot be separated by removing finitely many edges of $G$.

Another example arises when we glue together finite matroids along a tree $\Tcal$ of arbitrary overlap. In such cases we need the matroids to be representable over a common field in order to define the gluing: see~\cite{BC:psi_matroids} for details. Once more, if $\Psi$ is a Borel set of ends then we get a matroid $M_{\Psi}(\Tcal)$ in which we only allow circuits to use ends from $\Psi$. Once more, $M_{\emptyset}(\Tcal)$ is finitary and $M_{\Omega}(\Tcal)$ is cofinitary. If the tree is nice then $M_{\emptyset}(\Tcal)$ and $M_{\Omega}(\Tcal)$ are twinned. We conjecture that in such cases all tame matroids between $M_{\emptyset}(\Tcal)$ and $M_{\Omega}(\Tcal)$ are of the form $M_{\Psi}(\Tcal)$.

Even though there are lots of examples, the lattices of tame matroids between twinned pairs of matroids are poorly understood. For example, if $(M_f, M_c)$ is a twinned pair of matroids it is clear that every finite minor of $M_f$ on a subset $F$ of their common ground set also arises as a minor on the same subset $F$ of any matroid lying between them. It isn't clear whether the converse holds, even if $N$ has to be tame. Even the following question remains open:

\begin{oque}
We say a tame matroid is {\em binary} if it does not have $U_{2,4}$ as a minor~\cite{BC:rep_matroids}.
Let $(M_f, M_c)$ be a twinned pair of binary matroids. Must every tame matroid lying between them be binary?
\end{oque}

\bibliographystyle{plain}
\bibliography{literatur}

\end{document}